\documentclass[12pt]{article}

\usepackage[a4paper]{geometry}

\usepackage[utf8]{inputenc}
\usepackage[english]{babel}

\usepackage{biblatex}
\usepackage{amsmath}
\usepackage{amsthm}
\usepackage{amssymb}

\newtheorem{theorem}{Theorem}
\newtheorem{lemma}{Lemma}
\newtheorem{proposition}{Proposition}
\newtheorem{remark}{Remark}

\title{A Cone-preserving Solution to a Nonsymmetric Riccati Equation}
\author{Emil Vladu\thanks{The authors are with the Department of Automatic Control at Lund University. This work was partially supported by the Wallenberg AI, Autonomous Systems and Software Program (WASP) funded by the Knut and Alice Wallenberg Foundation, as well as the European Research
Council (Advanced Grant 834142).}, Anders Rantzer}
\date{}
\begin{document}
\maketitle

\begin{abstract}
    In this paper, we provide the following simple equivalent condition for a nonsymmetric Algebraic Riccati Equation to admit a stabilizing cone-preserving solution: an associated coefficient matrix must be stable. The result holds under the assumption that said matrix be cross-positive on a proper cone, and it both extends and completes a corresponding sufficient condition for nonnegative matrices in the literature. Further, key to showing the above is the following result which we also provide: in order for a monotonically increasing sequence of cone-preserving matrices to converge, it is sufficient to be bounded above in a single vectorial direction.
\end{abstract}

\section{Introduction}
\noindent Algebraic Riccati equations have been studied extensively in the literature over the years, e.g., \cite{lancaster1995algebraic} and the references therein. They often appear on the form $$XBX + DX + XA + C = 0,$$ where $A, B, C, D \in \mathbb{R}^{n \times n}$, and are not only of theoretical but also of practical interest. Often, as in control theory \cite{dullerud2013course}\cite{willems1971least}, $D$ is taken as $A^T$ with $B, C$ symmetric, and a symmetric solution $X$ is sought for. A typical result under the assumption of stabilizability and detectability may provide a unique stabilizing solution in certain settings, e.g., \cite{kuvcera1973review}. However, more recently there has been an increased interest in its nonsymmetric counterpart, see e.g., \cite{freiling2002survey} and the references therein.

The specific case in which the negative of the matrix
$$L = \begin{pmatrix}
    A && B \\
    C && D
\end{pmatrix}$$ forms an M-matrix, i.e., an anti-stable matrix with nonpositive offdiagonal elements, with applications in transport theory and Markov models, has attracted some interest, see e.g., \cite{guo2001nonsymmetric}\cite{guo2007iterative} and the references therein. In particular, part of \cite[Theorem 1.1]{guo2007iterative} provides the following sufficient condition for the existence of a stabilizing, entrywise nonnegative solution: $-L$ should be an M-matrix. Analogue statements for irreducible singular matrices and more recently also regular matrices \cite{guo2016algebraic} are known to hold.

The objective of the present paper is to both generalize the above result to proper cones as well as to complete it into an equivalence (Theorem \ref{thrm:cones}). Subsequent papers on the topic appear to be focused exclusively on the M-matrix case and its applications, with no mention about proper cones to the best of the authors' knowledge. By contrast, the main purpose of this paper is the desire to better understand the result by identifying the structure which generates it: the conic structure of the nonnegative orthant in $\mathbb{R}^n$. In addition, the main result turns out to be useful also in a different context, namely control theory \cite{vladu2024stability}. 

In order to prove the main result, a fixed-point iteration approach similar to the one in \cite{guo2001nonsymmetric} is used in one direction. However, rather than exploiting the Kronecker product, we shall pass through the analytical integral solution of the Sylvester equation. Further, and more importantly, convergence no longer becomes a simple matter of applying the monotone convergence theorem elementwise. As a result, we also provide a somewhat unexpected convergence result (Theorem \ref{thrm:mono}) for a sequence of monotonically increasing cone-preserving matrices with a vectorial upper bound. This is novel to the best of the authors' knowledge.

The outline of the paper is as follows: Section 2 reviews the crucial notion of cross-positivity on a proper cone as well as some elementary facts in cone theory. In Section 3, we present the results of the paper along with their proofs. Section 4 subsequently concludes the paper with some suggestions for future works.

\section{Preliminaries}
In this section, we clarify the notation used in the paper and provide the necessary background required for the results.

\subsection{Definitions and Notation}
In this subsection, we explain some definitions and notation used throughout the paper. $\mathbb{R}$, $\mathbb{R}^n$ and $\mathbb{R}^{n \times m}$ refer to the set of real numbers, $n$-dimensional vectors and $n \times m$ matrices with entries in $\mathbb{R}$, respectively. $\succeq_K$ refers to the partial order induced by a proper cone $K$, see Subsection 2.2. If $A \in \mathbb{R}^{n \times n}$, then $A$ is said to be stable if all its eigenvalues have negative real part. $\lVert A \rVert$ refers to the corresponding norm induced by an inner product, if such a function has been supplied, and otherwise to the spectral norm.

\subsection{Cone Theory}
In this subsection, we provide some necessary background on cone theory from the literature. For more results on this topic, see e.g., \cite{berman1994nonnegative}\cite{schneider2006matrices} and the references therein.

A set $K \subseteq \mathbb{R}^n$ is said to be a cone if $x \in K$ and $\alpha \geq 0$ imply $\alpha x \in K$. A convex, closed and pointed ($K \cap -K = \{0\}$) cone with non-empty interior is said to be proper. A proper cone $K$ induces a partial order $\succeq_K$ such that $x \succeq_K y$ if and only if $x - y \in K$. If $x - y$ lies in the interior of $K$, we say that $x \succ_K y$. The standard example of a proper cone is the nonnegative orthant in $\mathbb{R}^n$, and $\succeq_K$ then reduces to the usual inequality between real numbers applied entrywise. 

Given a proper cone $K \subseteq \mathbb{R}^n$, the associated dual cone is defined as $$K_* = \{y \in \mathbb{R}^n \mid y^T x \geq 0 \; \mathrm{for \; all} \; x \in K\}.$$ The interior of the dual cone consists of all those $y$ such that $y^T x > 0$ for all nonzero $x \in K$. If a cone is proper, the associated dual cone is also proper \cite[Chap. 2.6.1]{boyd2004convex}. Further, it is routine to verify from the definitions that $\big( K \times K \big)_* = K_* \times K_*$ and that $K \times K$ is a proper cone in $\mathbb{R}^{2n}$. 

A matrix $A \in \mathbb{R}^{n \times n}$ is said to be cross-positive on $K$ if $x \in K$ and $y \in K_*$ with $y^T x = 0$ imply $y^T A x \geq 0$. In particular, the set of cross-positive matrices on the nonnegative orthant is simply the set of matrices with nonnegative offdiagonal elements. For more on cross-positive matrices, see e.g., \cite{schneider1970cross}. 

A matrix $A \in \mathbb{R}^{n \times n}$ such that $AK \subseteq K$ for a proper cone $K \subseteq \mathbb{R}^n$ is said to be $K$-nonnegative or leave $K$ invariant. The set of such matrices is denoted $\pi (K)$ and is itself known to be a proper cone in $\mathbb{R}^{n \times n}$ \cite[Chap. 1.1]{berman1994nonnegative} in the vector sense after a standard identification with $\mathbb{R}^{n^2}$. Hence, if a proper cone $K \subseteq \mathbb{R}^n$ is specified, $X \succeq_K Y$ where $X, Y \in \mathbb{R}^{n \times n}$ means that $X - Y \in \pi (K)$, i.e., $X - Y$ is $K$-nonnegative. 

Next, we gather some elementary facts about $K$-nonnegativity and cross-positivity of which we shall make frequent use.
\begin{proposition} \label{prop:basic}
    Given $x, y \in \mathbb{R}^n$, $A, B \in \mathbb{R}^{n \times n}$ and a proper cone $K \subseteq \mathbb{R}^n$, the following holds:
    \begin{enumerate}
        \item[(i)] If $A$ and $B$ are $K$-nonnegative, then so is $A + B$ and $AB$.
        \item[(ii)] If $A$ is $K$-nonnegative and $x \succeq_K y$, then $Ax \succeq_K Ay$.
        \item[(iii)] If $A \succeq_K B$ and $x \succeq_K 0$, then $Ax \succeq_K Bx$.
        \item[(iv)] If $A$ is $K$-nonnegative, then $A$ is cross-positive on $K$.
        \item[(v)] If $A$ and $B$ are cross-positive on $K$, then so is $A + B$.
    \end{enumerate}
\end{proposition}
\begin{proof}
    Regarding (i), if $x \in K$, then $(A + B)x = Ax + Bx \succeq_K 0$ and $ABx = Ay \succeq_K 0$ with $y = Bx \succeq_K 0$ since $K$ is a proper cone. As for (ii), note that $Ax - Ay = A(x - y) \succeq_K 0$ so that $Ax \succeq_K Ay$, since $x - y \in K$ by assumption. Similarly, in (iii) we note that $A - B$ is $K$-nonnegative by assumption, and so $(A - B)x \succeq_K 0$, i.e., $Ax \succeq_K Bx$. Further, (iv) follows by definition of cross-positivity, as $z = Ax \succeq_K 0$ if $x \succeq_K 0$ by assumption, so that $y^T Ax = y^T z \geq 0$ if $y \in K_*$ by definition of the dual cone. Finally, if $x \in K$, $y \in K_*$ with $y^T x = 0$, we have $y^T(A + B)x = y^T A x + y^T Bx \geq 0$ by assumption and so $A + B$ is cross-positive, i.e., (v) holds.
\end{proof}
We now recall the following monotone convergence result.
\begin{lemma} \label{lem:conv} \cite[Lemma 1]{berman1974cones} Let $K \subseteq \mathbb{R}^n$ be a proper cone and let $\{s_i \}_{i = 1}^\infty$ be such that $s_i \preceq_K s_{i + 1}$. Let $t \in \mathbb{R}^n$ be such that $s_i \preceq_K t$ for every positive integer $i$. Then the sequence $\{s_i \}_{i = 1}^\infty$ converges.
\end{lemma}
We close this section with two important results on cross-positivity.
\begin{lemma} \label{lem:cross_eAt} \cite[Theorem 3]{schneider1970cross} Let $K \subseteq \mathbb{R}^n$ be a proper cone and $A \in \mathbb{R}^{n \times n}$. Then $A$ is cross-positive on $K$ if and only if $e^{At}$ is $K$-nonnegative for all $t \geq 0$.
\end{lemma}
\begin{lemma} \label{lem:cross_stable} \cite[Facts 7.1, 7.3, 7.5]{schneider2006matrices} Suppose $A  \in \mathbb{R}^{n \times n}$ is cross-positive on a proper cone $K \subseteq \mathbb{R}^n$. Then the following are equivalent:
\begin{enumerate}
    \item[(i)] $A$ is stable.
    \item[(ii)] There exists $x \succ_K 0$ such that $Ax \prec_K 0$.
    \item[(iii)] $-A^{-1}$ is $K$-nonnegative.
\end{enumerate}
\end{lemma}

\section{Results}
In this section, we provide the results of the paper. The main result is the following.
\begin{theorem} \label{thrm:cones}
    Let the proper cone $K \subseteq \mathbb{R}^n$ and the matrices $A, B, C, D \in \mathbb{R}^{n \times n}$ be given. Suppose now that $$L = \begin{pmatrix} A && B \\ C && D \end{pmatrix}$$ is cross-positive on $K \times K$. Then $L$ is stable if and only if
    \begin{equation} \label{eq:ric}
    XBX + DX + XA + C = 0
    \end{equation} has a solution $X_* \succeq_K 0$ such that $A + BX_*$ and $D + X_* B$ are stable and cross-positive on K.
\end{theorem}
\begin{proof}
    See below.
\end{proof}
\begin{remark}
    In the special case that $K$ is taken as the nonnegative orthant, the direction corresponding to sufficiency in Theorem \ref{thrm:cones} is equivalent to the part of the statement in \cite[Theorem 3.1]{guo2001nonsymmetric} corresponding to (nonsingular) M-matrices, see Section 1. This follows because cross-positive matrices on the nonnegative orthant have nonnegative offdiagonal elements and vice versa.
\end{remark}

The next result is not only instrumental in proving Theorem \ref{thrm:cones} but is also interesting in its own right.
\begin{theorem} \label{thrm:mono}
    Let the proper cone $K \subseteq \mathbb{R}^n$ and the sequence $\{X_i\}_{i=1}^\infty$ in $\mathbb{R}^{n \times n}$ be given. Suppose now that $0 \preceq_K X_{i} \preceq_K X_{i+1}$ and that there exist $s, r \in \mathbb{R}^n$ with $r \succ_K 0$ such that $X_i r \preceq_K s$ for all positive integers $i$. Then $\{X_i \}_{i=1}^\infty$ converges.
\end{theorem}
\begin{proof}
    Let $w \in \mathbb{R}^n$ be given and note that because $r \succ_K 0$, there must exist some $\varepsilon > 0$ such that both $r - \varepsilon w \succeq_K 0$ and $r + \varepsilon w \succeq_K 0$. Define now $a_i = X_i r$ and $b_i = X_i (r - \varepsilon w )$. First, it is clear from the assumptions and the fact that $r - \varepsilon w \succeq_K 0$ that $(X_{i + 1} - X_i ) r \succeq_K 0$ and $(X_{i + 1} - X_i ) (r - \varepsilon w) \succeq_K 0$, i.e., $a_{i} \preceq_K a_{i + 1}$ and $b_{i} \preceq_K b_{i + 1}$, respectively. Second, since by assumption $X_i \succeq_K 0$ and $X_i r \preceq_K s$, we have $a_i \preceq_K s$ and $$0 \preceq_K X_i (r + \varepsilon w) = X_i (2r - r + \varepsilon w) = 2 X_i r - X_i (r - \varepsilon w) = 2 X_i r - b_i$$ so that $b_i \preceq_K 2 X_i r \preceq_K 2s$. It follows by Lemma \ref{lem:conv} that both $\{a_i \}_{i=1}^\infty$ and $\{b_i \}_{i=1}^\infty$ converge. But then the sequence with elements $$X_i w = \frac{1}{\varepsilon} X_i (r - r + \varepsilon w) = \frac{1}{\varepsilon} (a_i - b_i )$$ must also converge. Now, since this holds for all $w \in \mathbb{R}^n$, one can choose $w$ so as to pick out each column of $X_i$ to show that they all converge. But by the equivalence of norms, this implies that $\{X_i \}_{i=1}^\infty$ converges entrywise and the conclusion follows.
\end{proof}

In order to prove Theorem \ref{thrm:cones}, we shall require the following additional lemmata.
\begin{lemma} \label{lem:lyap}
    Let the proper cone $K \subseteq \mathbb{R}^n$ and the matrices $A, C, D \in \mathbb{R}^{n \times n}$ be given. Suppose now that $A$ and $D$ are stable and cross-positive on $K$ and that $C \succeq_K 0$. Then 
    \begin{equation} \label{eq:lyap}
    DX + XA + C = 0
    \end{equation}
    has a unique solution $X_* \succeq_K 0$.
\end{lemma}
\begin{proof}
    Existence and uniqueness of a solution $X_*$ follows from Theorem 4.4.6 in \cite{horn1991topics}, as $A$ and $D$ are stable by assumption. Further,
    \begin{equation} \label{eq:lyap_sol}
    X_* = \int_0^\infty e^{D t}Ce^{A t}\mathrm{dt},
    \end{equation}
    as 
    \begin{equation*}
    \begin{aligned}
    -C &= \big [ e^{Dt}Ce^{At} \big ] _0^\infty = \int_0^\infty \frac{d}{dt} \big( e^{Dt}Ce^{At} \big ) \mathrm{dt} \\ &= \int_0^\infty \big ( De^{Dt}Ce^{At} + e^{Dt}Ce^{At}A \big ) \mathrm{dt} = DX_* + X_* A,
    \end{aligned}
    \end{equation*}
    similar to the well-known Lyapunov equation solution. It follows now from Lemma \ref{lem:cross_eAt} and Proposition \ref{prop:basic} (i) that the integrand in (\ref{eq:lyap_sol}) is $K$-nonnegative for all $t \geq 0$. As a result, we have $X_* \succeq_K 0$, as $\pi (K)$ is a proper cone and is therefore closed and preserves nonnegative linear combinations, see Section 2.
\end{proof}

\begin{lemma} \label{lem:kxk}
    Let the proper cone $K \subseteq \mathbb{R}^n$ and the matrices $A, B, C, D \in \mathbb{R}^{n \times n}$ be given. Then $$L = \begin{pmatrix} A && B \\ C && D \end{pmatrix}$$ is cross-positive on $K \times K$ if and only if $A$ and $D$ are cross-positive on $K$ and $B, C \succeq_K 0$.
\end{lemma}

\begin{proof}
     In the first direction, suppose that $L$ is cross-positive on $K \times K$ and take any $x \in K$ and $y \in K_*$ such that $y^T x = 0$. It follows that $\bar{y}^T \bar{x} = 0$, where $\bar{x} = (x^T, 0^T)^T \in K \times K$ and $\bar{y} = (y^T, 0^T)^T \in \big( K \times K \big)_*$, and so by assumption $y^T Ax = \bar{y}^T L \bar{x} \geq 0$, implying that $A$ is cross-positive on $K$. Taking instead $\bar{x} = (0, x^T)^T \in K \times K$ and noting that $\bar{y}^T \bar{x} = 0$ -- this time for any $x \in K$ and $y \in K_*$ -- again by cross-positivity $y^T B x = \bar{y}^T L \bar{x} \geq 0$. Since this holds for any fixed $x \in K$ as $y$ ranges over all elements in $K_*$, we must have $Bx \in K_{**}$, i.e., $B \succeq_K 0$ as for proper cones $K_{**} = K$, see e.g., \cite[Chap. 2.6.2]{boyd2004convex}. Similar reasoning gives that $D$ is cross-positive on $K$ and that $C \succeq_K 0$.
    
    As for the other direction, suppose that $y = (y_1 ^T , y_2 ^T )^T \in \big(K \times K \big)_*$ and $z = (z_1 ^T , z_2 ^T )^T \in K \times K$ are given such that $y^T z = 0$, i.e., 
     \begin{equation} \label{eq:dual}
     y^T z = y_1^T z_1 + y_2^T z_2 = 0.
     \end{equation} Note now that since $\big( K \times K \big)_* = K_* \times K_*$, we must have $y_1 ^T z_1 \geq 0$ and $y_2 ^T z_2 \geq 0$. As such, the only possibility for (\ref{eq:dual}) to hold is that both $y_1 ^T z_1 = 0$ and $y_2 ^T z_2 = 0$. Hence, since $A$ and $D$ are cross-positive on $K$ by assumption, it follows that $y_1^T A z_1 \geq 0$ and $y_2^T D z_2 \geq 0$. Consequently, $$y^T L z = y_1^T A z_1 + y_1^T B z_2 + y_2^T C z_1 + y_2^T D z_2 \geq 0$$ since by assumption $B, C \succeq_K 0$. But this means exactly that $L$ is cross-positive on $K \times K$.
\end{proof}

\begin{lemma} \label{lem:uv}
    Let the proper cone $K \subseteq \mathbb{R}^n$ and the matrices $A, B, C, D \in \mathbb{R}^{n \times n}$ be given. Suppose now that $$L = \begin{pmatrix} A && B \\ C && D \end{pmatrix}$$ is stable and cross-positive on $K \times K$. Then $A$ and $D$ are stable and there exist $u_1, u_2, v_1, v_2 \in \mathbb{R}^n$ with $v_1, v_2 \succ_K 0$ such that 
    \begin{equation} \label{eq:Lv}
        Av_1 + Bv_2 = u_1 \prec_K 0, \; \; \; \; Cv_1 + Dv_2 = u_2 \prec_K 0.
    \end{equation}
\end{lemma}
\begin{proof}
    Invoke immediately Lemma \ref{lem:cross_stable} to show that there exists a $v \succ_{K \times K} 0$ such that $Lv = u \prec_{K \times K} 0$. With the even partitioning $u = (u_1 ^T , u_2 ^T )^T$, $v = (v_1 ^T , v_2 ^T )^T$, we thus obtain (\ref{eq:Lv}) as desired. In order to see that $v_1, v_2 \succ_K 0$, note that since $v$ lies in the interior of $K \times K$, there is an $\varepsilon > 0$ such that for all $y \in \mathbb{R}^{2n}$ with $\lVert y - v \rVert < \varepsilon$ we have $y \in K \times K$. Consider now any $y_1 \in \mathbb{R}^n$ such that $\lVert y_1 - v_1 \rVert < \varepsilon$. Since $\lVert y_1 - v_1 \rVert = \lVert y - v \rVert$ with $y = (y_1 ^T , v_2 ^T ) ^T$, it follows that $y \in K \times K$ and so $y_1 \in K$. Hence, $v_1$ lies in the interior of $K$, i.e., $v_1 \succ_K 0$. Similar reasoning gives $v_2 \succ_K 0$.
    
    Finally, Lemma $\ref{lem:kxk}$ implies that $A$ and $D$ are cross-positive on $K$ and that $B, C \succeq_K 0$. Consequently, $Bv_2 \succeq_K 0$ and $Cv_1 \succeq_K 0$, and as a result, the above inequalities (\ref{eq:Lv}) imply that $Av_1 \prec_K 0$ and $Dv_2 \prec_K 0$, as $p \succ_K 0$ and $q \succeq_K 0$ imply $p + q \succ_K 0$. Another application of Lemma \ref{lem:cross_stable} therefore implies that $A$ and $D$ are both stable.
\end{proof}

We are now ready for the proof of Theorem \ref{thrm:cones}.

\begin{proof}
    \textbf{Theorem \ref{thrm:cones}} 
    
    \noindent Suppose in the first direction that $L$ is stable. The proof outline is then as follows: we construct a monotonically increasing sequence of matrices which is shown to converge to the desired solution of (\ref{eq:ric}). For this purpose, consider the recursion
    \begin{equation} \label{eq:rec}
        D X_{i+1} + X_{i+1} A = -X_{i} B X_{i} - C
    \end{equation}
    and set $X_0 = 0$. We proceed by induction to show that the sequence $\{ X_i \}_{i = 0}^\infty$ generated by (\ref{eq:rec}) is both well-defined and monotonically increasing. In the case that $i = 1$, an application of Lemma \ref{lem:lyap} directly gives a unique solution $X_1 \succeq_K 0 = X_0$. Suppose now that $X_{i}$ is well-defined through (\ref{eq:rec}) and that $X_{i} \succeq_K X_{i - 1}$ up until $i = k$ for some positive integer $k$. By Lemma \ref{lem:kxk}, $A$ and $D$ are both cross-positive on $K$ and $B, C \succeq_K 0$. Thus, $X_k B X_k + C \succeq_K 0$ by the induction assumption and so by Lemma \ref{lem:lyap}, $X_{k+1}$ follows uniquely from (\ref{eq:rec}), as $A$ and $D$ are stable by Lemma \ref{lem:uv}. Further, subtracting $DX_k + X_k A$ from both sides in (\ref{eq:rec}) yields 
    \begin{equation*}
    \begin{aligned}
    D(X_{k + 1} - X_k) + (X_{k + 1} - X_k)A &= -X_k B X_k - DX_k - X_k A - C \\ & \preceq_K -X_{k - 1} B X_{k - 1} - D X_k - X_k A - C = 0.
    \end{aligned}
    \end{equation*}
    Here, the induction assumption (\ref{eq:rec}) was invoked in the second equality, and $X_{k} \succeq_K X_{k-1}$ was invoked to give $X_k B X_k \succeq_K X_{k - 1} B X_{k - 1}$, hence the inequality. The latter statement follows from the assumption $B \succeq_K 0$, repeated application of Proposition \ref{prop:basic} (i) and transitivity. But again by Lemma \ref{lem:lyap}, this means that $X_{k+1} - X_k \succeq_K 0$, i.e., $X_{k+1} \succeq_K X_k$. Thus, the sequence $\{ X_i \}_{i = 0}^\infty$ generated by (\ref{eq:rec}) is well-defined and monotonically increasing by induction.

    Next, we show that $\{X_i\}_{i = 0}^\infty$ is bounded above in the sense that there exist $r, s \in \mathbb{R}^n$ with $r \succ_K 0$ such that $X_i r \preceq_K s$. For this purpose, apply Lemma \ref{lem:uv} to obtain (\ref{eq:Lv}), choose $r = v_1$ and $s = v_2 - D^{-1} u_2$, where $-u_1, -u_2, v_1, v_2 \succ_K 0$, and proceed by induction. The $i = 0$ case follows immediately: $s = v_2 - D^{-1}(Cv_1 + Dv_2) = -D^{-1}Cv_1 \succeq_K 0 = X_0 r$ by Proposition \ref{prop:basic} (i), since $-D^{-1} \succeq_K 0$ by Lemma \ref{lem:cross_stable} as $D$ is cross-positive on $K$ by assumption and stable from Lemma \ref{lem:uv}. Suppose now $X_{i} v_1 \preceq_K s = v_2 - D^{-1} u_2$ where $i = k$ for some $k \geq 0$. We have
    \begin{equation*}
    \begin{aligned}
        -DX_{k+1}v_1 &= X_{k + 1}Av_1 + X_k BX_k v_1 + Cv_1 \preceq_K X_{k + 1} Av_1 + X_k Bv_2 - Dv_2 + u_2 \\ & = (X_{k + 1} - X_{k})Av_1  + X_k u_1 - Dv_2 + u_2 \preceq_K -Dv_2 + u_2. 
    \end{aligned}
    \end{equation*}
    Here, the first equality comes from (\ref{eq:rec}). In the first inequality, the second expression in (\ref{eq:Lv}) and $X_k v_1 \preceq_K v_2$ were invoked, noting that $-D^{-1} u_2 \preceq_K 0$. In the second equality, the first expression in (\ref{eq:Lv}) was used, and in the second inequality monotonicity was invoked along with the fact that $Av_1 = u_1 - Bv_2 \prec_K 0$ and $X_k u_1 \preceq_K 0$. Multiplication by $-D^{-1} \succeq_K 0$ thus gives $X_{k + 1} v_1 \preceq_K v_2 - D^{-1} u_2$, and so the desired conclusion follows by induction.

    Altogether, we have shown that $0 = X_0 \preceq_K X_i \preceq_K X_{i + 1}$ and that there exist $r, s \in \mathbb{R}^{n}$ with $r \succ_K 0$ such that $X_i r \preceq_K s$ for all $i \geq 0$. An application of Theorem \ref{thrm:mono} thus shows that $\{X_i \}_{i = 1}^\infty$ is convergent, i.e., $X_i \rightarrow X_*$ as $i \rightarrow \infty$ for some $X_* \in \mathbb{R}^{n \times n}$. Consequently, since the sequence satisfies (\ref{eq:rec}), we have $$X_* BX_* + DX_* + X_* A + C = 0.$$ Further, since $\pi (K)$ is a proper cone and hence closed, see Section 2, $\{X_i \}_{i = 1}^\infty$ converges inside $\pi (K)$, i.e., $X_* \succeq_K 0$.

    It remains to show that $A + BX_*$ and $D + X_* B$ are stable and cross-positive on $K$. For this purpose, note that $X_* v_1 \preceq_K v_2$ because $K$ is closed. Together with the assumption $B \succeq_K 0$ and Proposition \ref{prop:basic} (ii), we thus have $$(A + BX_*)v_1 \preceq_K Av_1 + Bv_2 = u_1 \prec 0,$$ where in the final equality the first expression in (\ref{eq:Lv}) was invoked. But $A + BX_*$ is cross-positive on $K$ according to Proposition \ref{prop:basic} (i), (iv) and (v), since $A$ is also cross-positive on $K$ by assumption. It thus follows from Lemma \ref{lem:cross_stable} that $A + BX_*$ is stable.
    
    As for $D + X_* B$, note that the matrix $$\begin{pmatrix} D^T && B^T \\ C^T && A^T \end{pmatrix} = \begin{pmatrix} 0 && I \\ I && 0 \end{pmatrix} \begin{pmatrix} A && B \\ C && D \end{pmatrix}^T \begin{pmatrix} 0 && I \\ I && 0 \end{pmatrix}$$ has the same spectrum as $L$ and is therefore stable, as it is a similarity transform of $L^T$. Further, it is cross-positive on $K_* \times K_*$ by Lemma \ref{lem:kxk}, as $D^T$ and $A^T$ are cross-positive on $K_*$ due to $K_{**} = K$, and $B^T, C^T \succeq_{K_*} 0$ by definition of $K_*$. By the part of Theorem \ref{thrm:cones} proved so far, the equation $$ZB^T Z + A^T Z + ZD^T + C^T = 0$$ must therefore have a solution $Z_* \succeq_{K_*} 0$ such that $D^T + B^T Z_*$ is stable and cross-positive on $K_*$. Additionally, such a solution can be taken as the limit of a sequence $\{Z_i\}_{i=0}^{\infty}$ with $Z_0 = 0$ which satisfies the corresponding recursion (\ref{eq:rec}), i.e., $$A^T Z_{i + 1} + Z_{i + 1}D^T = -Z_i B^T Z_i - C^T .$$ But according to first part of the proof, this recursion is also satisfied by $\{X_i^T\}_{i = 0}^{\infty}$ with $X_0 = 0$ once (\ref{eq:rec}) is transposed. Because the sequence generated by the recursion was shown to be unique, we may conclude that $Z_i = X_i ^T$ and so by the continuity of the transpose operator $Z_* = X_* ^T$. But then the stability and cross-positivity of $D^T + B^T Z_* = D^T + B^T X_* ^T$ on $K_*$ implies the stability and cross-positivity of $D + X_* B = (D^T + B^T X_* ^T) ^T$ on $K$. This concludes the sufficiency part of the proof.
    
    As for necessity, note first that by the following well-known similarity transformation we have 
    \begin{equation*} \begin{aligned} 
    &\begin{pmatrix}
        I && 0 \\ -X_* && I
    \end{pmatrix} 
    \begin{pmatrix}
        A && B \\ -C && -D
    \end{pmatrix}
    \begin{pmatrix}
        I && 0 \\ X_* && I
    \end{pmatrix} \\ &= 
    \begin{pmatrix}
        A + BX_* && B \\ -X_* BX_* - DX_* - X_* A - C && -(D + X_* B)
    \end{pmatrix}\end{aligned}\end{equation*}.
    
    Consequently, the assumptions yield
    \begin{equation*}
    \begin{aligned}
        &-L^{-1} = - \Bigg( \begin{pmatrix}
            I && 0 \\ 0 && -I
        \end{pmatrix}
        \begin{pmatrix}
            A && B \\ -C && -D
        \end{pmatrix} \Bigg)^{-1} \\ &=
        - \Bigg( \begin{pmatrix}
            I && 0 \\ 0 && -I
        \end{pmatrix}
        \begin{pmatrix}
            I && 0 \\ X_* && I
        \end{pmatrix}
        \begin{pmatrix}
            A + BX_* && B \\ 0 && -(D + X_* B)
        \end{pmatrix} 
        \begin{pmatrix}
            I && 0 \\ -X_* && I
        \end{pmatrix} \Bigg)^{-1} \\ &=
        \begin{pmatrix}
        I && 0 \\ X_* && I
        \end{pmatrix}
        \begin{pmatrix}
            PX_* -(A + BX_* )^{-1} && P \\ -(D + X_* B)^{-1} X_* && -(D + X_* B)^{-1}
        \end{pmatrix} \succeq_{K \times K} 0
    \end{aligned}
    \end{equation*}
    where $P = (A + BX_* )^{-1} B (D + X_* B)^{-1}$. Here, the fact that $X_*$ solves $XBX + DX + XA + C = 0$ was invoked in the second equality. Note also that $A + BX_*$ and $D + X_* B$ being cross-positive and stable by assumption imply that $-(A + BX_* )^{-1} \succeq_K 0$ and $-(D + X_* B) ^{-1} \succeq_K 0$ by Lemma \ref{lem:cross_stable}, and so $P \succeq_K 0$ since $B \succeq_K 0$. Further, it is clear that if all four matrices in a $2 \times 2$ block matrix are $K$-nonnegative, then the block matrix is $K \times K$-nonnegative. As such, since also $X_* \succeq_K 0$ by assumption, the final inequality follows. But since it is assumed that $L$ is cross-positive on $K \times K$, another application of Lemma \ref{lem:cross_stable} implies that $L$ is stable. This concludes the proof.
\end{proof}

\section{Conclusions}
In this paper, the following equivalent condition is supplied for a nonsymmetric algebraic Riccati equation to admit a stabilizing cone-preserving solution: an associated coefficient matrix should be stable. This extends and completes an already published sufficient condition on the nonnegative orthant into an equivalence for general proper cones. Many additional properties, such as the minimality of the cone-preserving solution, follow from the stability of the aforementioned coefficient matrix in the well-studied nonnegative case. While this lies beyond the scope of the present paper, establishing how well these features generalize to proper cones would be interesting for future works. 

\section{Acknowledgements}
The authors are grateful to Dr. Dongjun Wu for proofreading the manuscript and providing helpful comments.

\printbibliography

\end{document}